\documentclass[11pt]{article}
\usepackage{amsmath,amssymb,amsthm,mathtools,enumitem}
\usepackage[a4paper,margin=1in]{geometry}
\usepackage{float}

\newtheorem{theorem}{Theorem}
\newtheorem{lemma}{Lemma}

\newcommand{\R}{\mathbb R}
\newcommand{\eps}{\varepsilon}

\title{An Upper Bound for Discrete Isometric Filling of Cycles}
\author{Runtai He}
\date{}

\begin{document}
\maketitle
\begin{abstract}
We study the discrete graph-metric analogue of Gromov's filling area problem for the cycle graph \(C_n\).  An abstract triangulation \(K\) is an isometric filling of \(C_n\) if \(\partial K=C_n\) and the graph distance between any two boundary vertices is not shortened inside the \(1\)-skeleton of \(K\).  Let \(D(n;\epsilon)\) denote the minimum number of vertices in a \((1-\epsilon)\)-Lipschitz filling of \(C_n\), and set
\[
D^*=\liminf_{\epsilon\to0^+}\liminf_{n\to\infty}\frac{D(n;\epsilon)}{n^2}.
\]
Previous work gives the general lower bound \(D^*\ge 1/8\), while discretizing the hemisphere gives the upper bound
\[
D^*\le \frac{1}{\pi\sqrt3}.
\]
In this paper we give an explicit discrete construction which improves the hemispherical upper bound.  More precisely, we construct isometric fillings \(K_n\) of \(C_n\) with
\[
|V(K_n)|\le \left(\frac16+o(1)\right)n^2,
\]
and hence
\[
D^*\le \frac16<\frac{1}{\pi\sqrt3}.
\]
This can directly illustrate the discrete filling area problem is a proper relaxation of Gromov's original filling area problem and cannot be used to settle Gromov's conjecture.
The construction is a concentric annular filling.
\end{abstract}

\section{Introduction}

Gromov's filling area problem asks for the least possible area of a compact Riemannian surface which fills a prescribed boundary curve without introducing shortcuts between boundary points.\cite{Gromov1983}  In its classical form, the boundary is the Riemannian circle of circumference \(2\pi\), and the conjectural minimizer is the hemisphere.  More precisely, if \(M\) is a compact orientable Riemannian surface with boundary \(S^1\) and
\[
d_M(x,y)=d_{S^1}(x,y)
\qquad
\text{for all }x,y\in S^1,
\]
then Gromov's filling area conjecture predicts that
\[
\operatorname{Area}(M)\ge 2\pi.
\]
The hypothesis is not merely topological: it requires that the filling introduce no shortcut between boundary points.  For instance, the Euclidean disk fills the circle topologically, but it does not fill it isometrically, since antipodal points become much closer through the disk than along the boundary.

A natural combinatorial relaxation replaces metric surfaces by abstract triangulations and replaces intrinsic length by graph distance in the \(1\)-skeleton.\cite{BriggsWells2026}  Thus an abstract triangulation \(K\) with boundary \(C_n\) is an isometric filling of \(C_n\) if
\[
d_K(x,y)=d_{C_n}(x,y)
\qquad
\text{for all }x,y\in V(C_n).
\]
Equivalently, the \(1\)-skeleton of \(K\) contains no shortcut between boundary vertices.  One may also consider approximate fillings: for \(0<\delta\le1\), a \(\delta\)-Lipschitz filling of \(C_n\) is an abstract triangulation \(K\) with \(\partial K=C_n\) such that
\[
d_K(x,y)\ge \delta\, d_{C_n}(x,y)
\qquad
\text{for all }x,y\in V(C_n).
\]
The asymptotic filling density is then captured by
\[
D^*
=
\liminf_{\epsilon\to0^+}\liminf_{n\to\infty}
\frac{D(n;\epsilon)}{n^2},
\]
where \(D(n;\epsilon)\) is the minimum number of vertices in a \((1-\epsilon)\)-Lipschitz filling of \(C_n\).

Briggs and Wells proved that every \(\delta\)-Lipschitz filling of \(C_n\) has at least
\[
\frac{\delta^3}{8}(n-1)^2+\frac12(n-1)
\]
vertices.\cite{BriggsWells2026}  In particular, for isometric fillings this gives the lower bound
\[
D^*\ge \frac18.
\]
On the other hand, a balanced triangular discretization of the hemisphere gives
\[
D^*\le \frac{1}{\pi\sqrt3}.
\]
Thus before the present construction the natural range was
\[
\frac18\le D^*\le \frac{1}{\pi\sqrt3}.
\]
The upper bound coming from the hemisphere is especially suggestive because of the continuous filling area conjecture.  If the discrete problem were an exact combinatorial counterpart of the Gromov's problem, one might expect the hemispherical constant to be optimal.  The purpose of this paper is to show that this is not the case for the graph-metric model.

Our main result is the following.

\begin{theorem}\label{thm:main}
\[
D^*\le \frac16.
\]
In particular,
\[
\frac16<\frac{1}{\pi\sqrt3}.
\]
Thus the discrete graph-metric filling problem admits fillings asymptotically smaller than the upper bound obtained by discretizing the hemisphere with a triangular grid. This confirms the expectation raised in \cite{BriggsWells2026} that the discrete problem is a proper relaxation of the continuous filling area problem.
\end{theorem}

The proof is constructive.  For each sufficiently large \(n\), we build an isometric filling \(K_n\) of \(C_n\) whose vertex count is asymptotic to \(n^2/6\).  The construction consists of three parts.

First, we place a protective collar of \(\rho n\) equal-length annuli near the boundary, where \(\rho>0\) is fixed during the proof and later allowed to tend to \(0\).  This collar is not responsible for the final constant; rather, it provides a positive error margin in the distance estimate.

Second, we insert the main annular region.  Its cycle lengths are chosen according to the square-root profile
\[
m(t)\approx n q(t),
\qquad
q(t)=\sqrt{1-4t},
\qquad
0\le t\le t_\eta,
\]
where \(t\) denotes depth divided by \(n\), and \(t_\eta=(1-\eta^2)/4\) is the stopping time at which \(q(t_\eta)=\eta\).  The square-root profile is the key point of the construction.  It comes from balancing two competing effects.  If a path travels inward to depth \(tn\), it pays a radial cost of approximately \(2tn\).  At the same time, inner cycles are shorter, so horizontal travel becomes cheaper; additionally, slanted edges between consecutive cycles may produce some angular displacement ``for free.''  The profile \(q(t)=\sqrt{1-4t}\) is exactly the critical profile for which this possible saving is balanced by the radial cost.

Third, once the inner cycle has length about \(\eta n\), we close it by a cone.  This cone cap creates very short paths inside the innermost cycle, but it is placed sufficiently deep that any boundary-to-boundary path using it must already pay more than \(n/2\) in radial cost.  Since \(n/2\) is the largest possible distance between two vertices of \(C_n\), the cone cap cannot create a boundary shortcut.

A central technical issue is to control angular drift.  We assign to every cycle an auxiliary circular coordinate of circumference \(n\), so that one edge of the outer boundary has coordinate length \(1\).  Equal-length annuli are triangulated with a half-step offset; hence crossing such an annulus once changes the circular coordinate by at most \(n/(2m)\), where \(m\) is the cycle length.  Shrinking annuli are triangulated by a staircase construction, and crossing such an annulus once changes the coordinate by at most \(n/M\), where \(M\) is the inner cycle length.  These local bounds imply a global drift estimate for arbitrary paths.  If a path from \(x\) to \(y\) reaches deepest layer \(C_h\), then it must pay a radial cost \(2h\), and after subtracting the maximum angular displacement contributed by slanted edges, the remaining angular separation must be paid by horizontal edges on cycles of length at least \(m_h\).

The main region is not made to shrink at every layer.  Instead, we approximate the square-root profile by a step function.  Each block consists of many equal-length annuli, and only \(B=\lceil \sqrt n\rceil=o(n)\) transition annuli perform actual shrinking.  This sparse shrinking is another essential feature of the proof.  Equal-length annuli have the optimal half-step drift bound, while transition annuli have a worse one-step drift bound.  Since the number of transition annuli is \(o(n)\), their total contribution to the drift is \(o(n)\) and is absorbed by the protective collar.

The vertex count follows from the same profile.  The protective collar contributes \(\rho n^2+o(n^2)\) vertices.  The main region contributes
\[
n^2\int_0^{t_\eta} q(t)\,dt+o(n^2).
\]
For \(q(t)=\sqrt{1-4t}\), this integral is
\[
\int_0^{t_\eta}\sqrt{1-4t}\,dt
=
\frac16(1-\eta^3).
\]
The cone cap contributes only lower-order terms.  Therefore, for fixed \(\rho\) and \(\eta\),
\[
\limsup_{n\to\infty}\frac{|V(K_n)|}{n^2}
\le
\rho+\frac16(1-\eta^3).
\]
Choosing \(0<\eta<\sqrt{\rho}\) and then letting \(\rho\to0\) gives the asserted upper bound \(D^*\le1/6\).

The rest of the paper is organized as follows.  Section~1 records the coordinate convention and notation used throughout the proof.  Section~2 constructs the two elementary annuli: equal-length annuli and shrinking annuli.  Section~3 proves the path drift estimate.  Section~4 gives the full construction.  Section~5 establishes the uniform estimates needed for the stepwise square-root profile.  Section~6 proves the core analytic inequality.  Section~7 verifies that the construction is isometric on the boundary.  Finally, Section~8 counts vertices and completes the proof of the theorem.
\section{Coordinate Convention and Notation}
The proof below avoids normalized angular variables.  Instead, all circular displacements are measured in the same units as the outer boundary cycle.  Namely, the full turn has length $n$, not $1$ and not $2\pi$.  Thus an edge of the boundary cycle $C_n$ has circular length $1$.  This convention keeps the distance estimates in edge-count units.

More precisely, all auxiliary coordinates take values in the circle
\[
\R/n\mathbb Z.
\]
The auxiliary coordinates are chosen once and for all as the cycles are created. 
They are not chosen independently for different annuli.

More precisely, suppose that the \(r\)-th cycle \(C_r\) has length \(m_r\), and write its vertices in cyclic order as
\[
W^{(r)}_0,W^{(r)}_1,\ldots,W^{(r)}_{m_r-1}.
\]
We assign a number \(\alpha_r\in \mathbb R/n\mathbb Z\), called the phase of \(C_r\), so that
\[
\Theta(W^{(r)}_i)=\alpha_r+\frac{ni}{m_r}\pmod n.
\]
Thus the vertices of \(C_r\) are equally spaced in the auxiliary circle of circumference \(n\).

The phase is determined recursively.  We start with \(\alpha_0=0\) on the boundary cycle \(C_0=C_n\).  If an equal-length annulus of length \(m\) is inserted between \(C_r\) and \(C_{r+1}\), then the inner cycle is placed halfway between consecutive vertices of \(C_r\), so
\[
\alpha_{r+1}=\alpha_r+\frac{n}{2m}.
\]
If a shrinking annulus from length \(m\) to length \(M\) is inserted, then the inner cycle is given the same phase:
\[
\alpha_{r+1}=\alpha_r.
\]
In this way, every cycle receives its coordinates exactly once, at the moment it is created. 
These coordinates are only an auxiliary device for estimating circular displacement along edges; they are not an additional metric on the two-dimensional complex.
For two points $a,b\in\R/n\mathbb Z$, let
\[
\|a-b\|_n
\]
denote the shorter circular distance on the circle of circumference $n$.

\section{Two elementary annuli}

\subsection{Equal-length annuli}

Let the outer cycle be
\[
U_0,U_1,\dots,U_{m-1},
\]
and assume that its coordinates have already been fixed in the recursive sense above:
\[
\Theta(U_i)=\alpha+\frac{ni}{m}\pmod n.
\]
We now create a new inner $m$-cycle
\[
V_0,V_1,\dots,V_{m-1}
\]
and assign its coordinates halfway between consecutive outer vertices:
\[
\Theta(V_i)=\alpha+\frac{n(i+1/2)}{m}\pmod n.
\]
Equivalently, the new inner cycle has phase
\[
\alpha' = \alpha+\frac{n}{2m},
\]
so its coordinates are again of the recursive form
\[
\Theta(V_i)=\alpha'+\frac{ni}{m}\pmod n.
\]
Add the triangles
\[
(U_i,U_{i+1},V_i),\qquad (U_{i+1},V_i,V_{i+1}),
\]
where indices are taken modulo $m$.  This gives a triangulated annulus from an $m$-cycle to an $m$-cycle.  Each slanted edge has circular displacement at most
\[
\frac{n}{2m}.
\]
For example,
\[
\left\|\Theta(U_i)-\Theta(V_i)\right\|_n
=\frac{n}{2m},
\qquad
\left\|\Theta(U_{i+1})-\Theta(V_i)\right\|_n
=\frac{n}{2m}.
\]

\begin{figure}[H]
\centering
\includegraphics[width=0.6\textwidth]{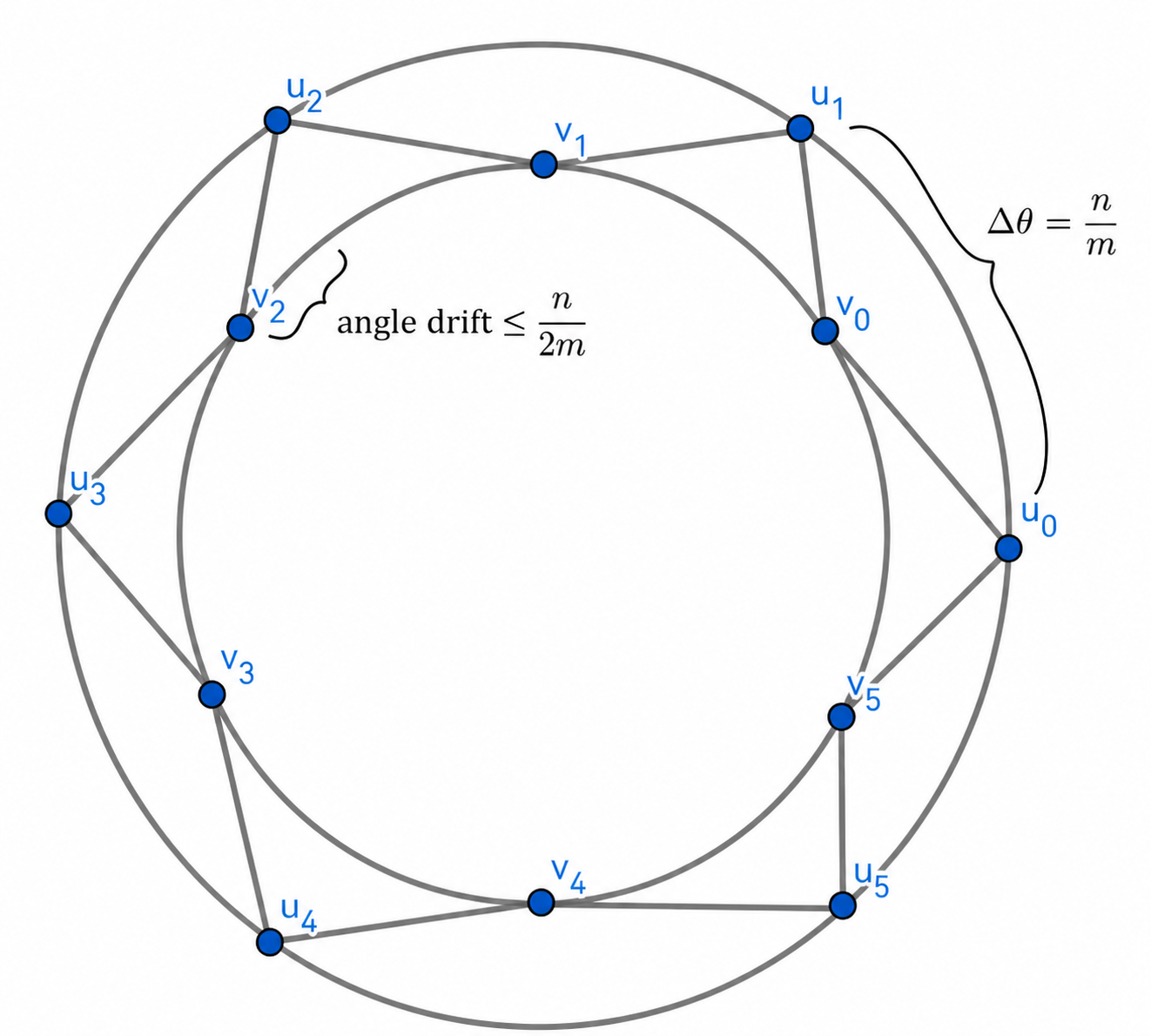}
\caption{Equal-length annulus}
\end{figure}

\subsection{Shrinking annuli}

We next give the corresponding construction when the inner cycle is shorter.

\begin{lemma}[Shrinking annulus triangulation]\label{lem:shrinking-annulus}
Let $m\ge M\ge3$.  Suppose the outer $m$-cycle
\[
U_0,U_1,\dots,U_{m-1}
\]
has already been assigned coordinates
\[
\Theta(U_i)=\alpha+\frac{ni}{m}\pmod n.
\]
Then one can create a new inner $M$-cycle
\[
V_0,V_1,\dots,V_{M-1}
\]
and triangulate the annulus between the two cycles so that, with the inner coordinates defined by the same phase
\[
\Theta(V_j)=\alpha+\frac{nj}{M}\pmod n,
\]
every slanted edge has circular displacement at most
\[
\frac{n}{M}.
\]
\end{lemma}

\begin{proof}
The inner cycle and its coordinates are part of the construction; they are not a reassignment of a previously coordinated cycle.  The phase of the new inner cycle is the same phase $\alpha$ as that of the outer cycle.  Cut the annulus open along a radial line.  Work temporarily with lifted vertices
\[
U_0,U_1,\dots,U_m,
\qquad
V_0,V_1,\dots,V_M,
\]
and later identify $U_m\sim U_0$ and $V_M\sim V_0$.  Define
\[
k_i=\left\lfloor \frac{Mi}{m}\right\rfloor,
\qquad i=0,1,\dots,m.
\]
Then
\[
k_0=0,
\qquad
k_m=M,
\qquad
k_{i+1}-k_i\in\{0,1\}.
\]
For each $i=0,\dots,m-1$, do the following.

If $k_{i+1}=k_i$, add the triangle
\[
(U_i,U_{i+1},V_{k_i}).
\]
If $k_{i+1}=k_i+1$, add the two triangles
\[
(U_i,U_{i+1},V_{k_i+1}),
\qquad
(U_i,V_{k_i},V_{k_i+1}).
\]
After gluing the two ends back together, $V_M$ is interpreted as $V_0$ and $U_m$ as $U_0$.

We check that this gives a valid triangulated annulus.  Each outer edge $(U_i,U_{i+1})$ occurs exactly once, namely in the step indexed by $i$.  Since the sequence $k_i$ starts at $0$, ends at $M$, and increases by $0$ or $1$ at each step, for every $j=0,\dots,M-1$ there is a unique $i$ such that
\[
k_i=j,
\qquad
k_{i+1}=j+1.
\]
Thus each inner edge $(V_j,V_{j+1})$ occurs exactly once.  In the cut-open strip, the listed triangles form an ordinary planar triangulation with no crossings and no gaps.  After the two ends are identified, the only boundary edges are the outer and inner cycle edges; every slanted edge is internal and is incident to exactly two triangles.  Since $M\ge3$, none of the listed triangles is degenerate.  Distinct triangles contain a unique outer edge or a unique inner edge, so no repeated triangle occurs.

It remains to prove the displacement estimate.  From
\[
k_i\le \frac{Mi}{m}<k_i+1
\]
we get
\[
\frac{nk_i}{M}\le \frac{ni}{m}<\frac{n(k_i+1)}{M}.
\]
Therefore $U_i$ has circular distance at most $n/M$ from both $V_{k_i}$ and $V_{k_i+1}$.  If $k_{i+1}=k_i$, the same inequality applied to $i+1$ shows that $U_{i+1}$ has circular distance at most $n/M$ from $V_{k_i}$.  If $k_{i+1}=k_i+1$, then $U_{i+1}$ has circular distance at most $n/M$ from $V_{k_i+1}$.  These are precisely the slanted edges appearing in the construction.  The endpoint gluing is harmless because the distance is taken in $\R/n\mathbb Z$.
\end{proof}

\begin{figure}[H]
\centering
\includegraphics[width=0.6\textwidth]{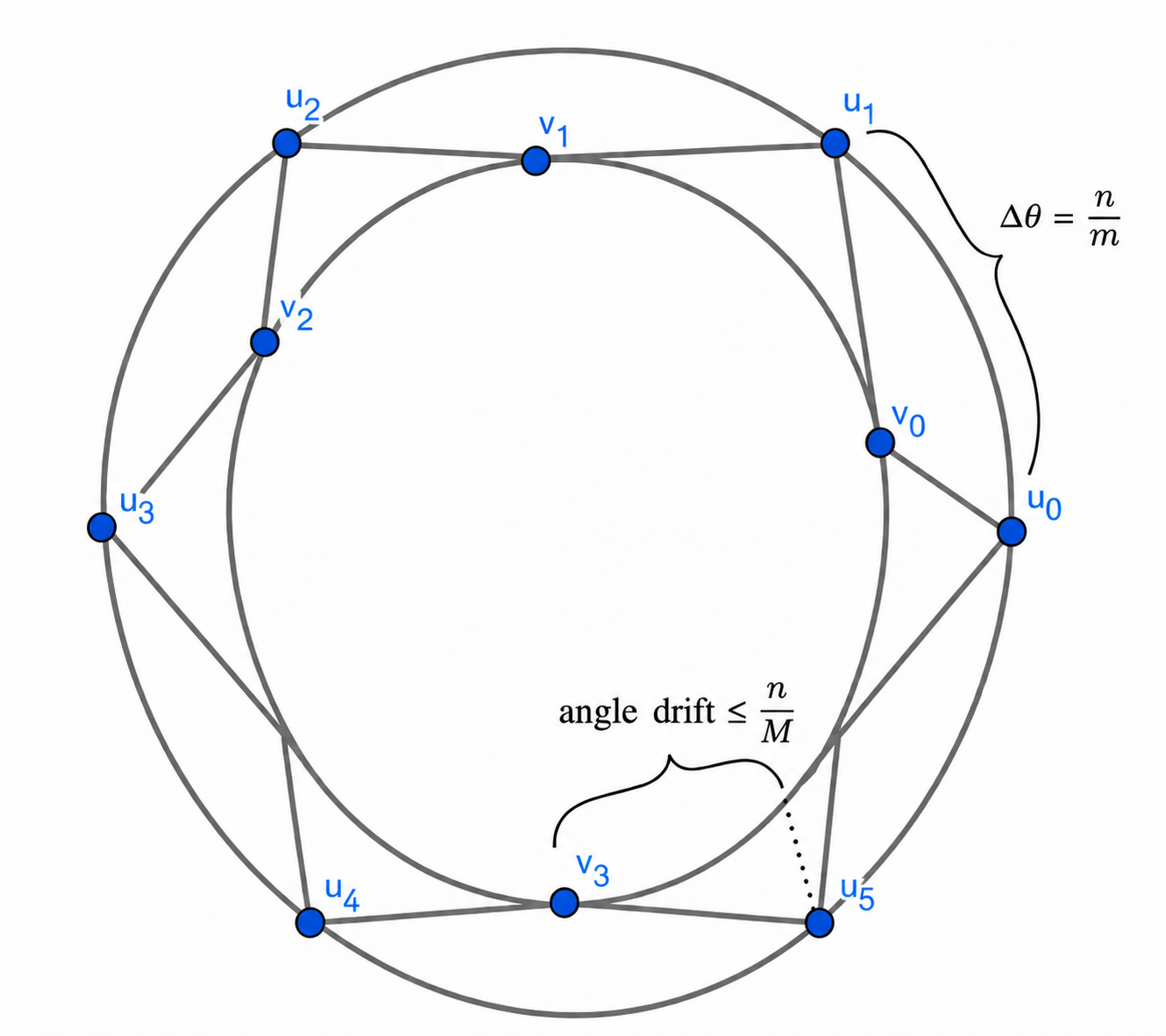}
\caption{Shrinking annulus}
\end{figure}

\section{A drift estimate for edge paths}

Consider concentric cycles
\[
C_0,C_1,\dots,C_J,
\]
where
\[
|C_r|=m_r,
\qquad
m_0=n,
\qquad
m_0\ge m_1\ge\cdots\ge m_J.
\]
The auxiliary coordinates on the cycles are chosen recursively as in the two elementary annuli above.  Equivalently, each cycle $C_r$ has a phase $\alpha_r$ such that its cyclically ordered vertices have coordinates $\alpha_r+ni/m_r\pmod n$.  Equal-length annuli update the phase by a half-step, and shrinking annuli keep the phase unchanged.  Thus every cycle is assigned a single coordinate system in $\R/n\mathbb Z$, and a horizontal edge in $C_r$ has circular length
\[
\frac{n}{m_r}.
\]

Between $C_r$ and $C_{r+1}$ we use an equal-length annulus if $m_{r+1}=m_r$, and the shrinking annulus from Lemma \ref{lem:shrinking-annulus} if $m_{r+1}<m_r$.  Define the one-crossing drift bound
\[
b_r=
\begin{cases}
\dfrac{n}{2m_r},& m_{r+1}=m_r,\\[6pt]
\dfrac{n}{m_{r+1}},& m_{r+1}<m_r.
\end{cases}
\]
Thus crossing the $r$-th annulus once changes the auxiliary coordinate by at most $b_r$.

\begin{lemma}[Path drift estimate]\label{lem:drift}
Let $x,y\in C_0$, and set
\[
L=d_{C_0}(x,y),
\qquad
0\le L\le \frac n2.
\]
Let $P:x\to y$ be an edge path whose deepest visited layer is $C_h$.  Define
\[
\mathcal A_h=2\sum_{r=0}^{h-1}b_r.
\]
Then
\[
|P|\ge 2h+\frac{m_h}{n}\,(L-\mathcal A_h)_+,
\]
where $(u)_+=\max\{u,0\}$.
\end{lemma}

\begin{proof}
Let $H$ be the total number of horizontal edges in $P$.  Let $c_r$ be the number of slanted edges of $P$ crossing the annulus $C_r\leftrightarrow C_{r+1}$.  Since $P$ starts and ends on $C_0$ and reaches $C_h$, it must cross each annulus $C_r\leftrightarrow C_{r+1}$, $0\le r<h$, at least twice.  Write
\[
c_r=2+e_r,
\qquad e_r\ge0.
\]
Then
\[
|P|=H+2h+
\sum_{r=0}^{h-1}e_r.
\]

The endpoints $x,y$ have circular separation $L$ in $\R/n\mathbb Z$.  Therefore the sum of the circular displacements of the edges of $P$ is at least $L$.  A horizontal edge in $C_r$ contributes at most $n/m_r\le n/m_h$, since $m_r\ge m_h$ for $r\le h$.  A slanted edge crossing the $r$-th annulus contributes at most $b_r\le n/m_h$.  Hence
\[
L
\le
\frac{n}{m_h}H
+2\sum_{r=0}^{h-1}b_r
+\frac{n}{m_h}\sum_{r=0}^{h-1}e_r.
\]
That is,
\[
L\le \frac{n}{m_h}H+\mathcal A_h+\frac{n}{m_h}\sum_{r=0}^{h-1}e_r.
\]
Rearranging gives
\[
H+\sum_{r=0}^{h-1}e_r
\ge
\frac{m_h}{n}(L-\mathcal A_h).
\]
Taking the positive part on the right and substituting into the expression for $|P|$ proves the lemma.
\end{proof}

\section{The construction}

Fix arbitrary $\rho>0$.  Choose $\eta>0$ such that
\[
0<\eta<1,
\qquad
\eta^2<\rho.
\]
All error terms in the distance proof below are taken as $n\to\infty$ with $\rho$ and $\eta$ fixed.

Set
\[
q(t)=\sqrt{1-4t},
\qquad
 t_\eta=\frac{1-\eta^2}{4},
\]
so that
\[
q(t_\eta)=\eta.
\]
Let
\[
w=\lceil \rho n\rceil,
\qquad
B=\lceil \sqrt n\rceil,
\qquad
\Delta_n=\frac{t_\eta}{B}.
\]
Divide $[0,t_\eta]$ into $B$ equal subintervals by
\[
t_b=b\Delta_n,
\qquad b=0,1,\dots,B.
\]
Define
\[
M_b=\left\lceil n q(t_b)\right\rceil.
\]
Since $q(0)=1$, we have $M_0=n$ exactly.  Also $M_B=\lceil n\eta\rceil$, so for all sufficiently large $n$ every cycle used in the main region has length at least $3$.

For $b=0,\dots,B-1$, define
\[
L_b=\left\lfloor n\Delta_n\right\rfloor.
\]
It should be noted that although we use the subscript $b$ in $L_b$, this does not imply that it is related to the index $b$. In fact, it is only related to $n$ and $\eta$. The use of the subscript $b$ is intended to indicate the Riemann sum structure that will be utilized subsequently.

The complex $K_n$ is built as follows.

\begin{enumerate}[label=(\arabic*)]
\item \textbf{Protective collar.}  Starting from the boundary cycle $C_0=C_n$, insert $w$ equal-length annuli, each of length $n$.

\item \textbf{Main region.}  For each $b=0,\dots,B-1$, insert $L_b$ equal-length annuli of length $M_b$.  Then insert one transition annulus from length $M_b$ to length $M_{b+1}$: if $M_{b+1}<M_b$, use Lemma \ref{lem:shrinking-annulus}; if $M_{b+1}=M_b$, use an equal-length annulus.

\item \textbf{Cone cap.}  The final inner cycle has length $M_B=\lceil n\eta\rceil$.  If it is written as
\[
Z_0,Z_1,\dots,Z_{M_B-1},
\]
add one central vertex $o$ and the triangles
\[
(o,Z_i,Z_{i+1}),
\qquad i=0,1,\dots,M_B-1,
\]
with indices modulo $M_B$.
\end{enumerate}

This produces an abstract triangulation whose boundary is exactly the original outer cycle $C_n$.

\begin{figure}[H]
\centering
\includegraphics[width=0.8\textwidth]{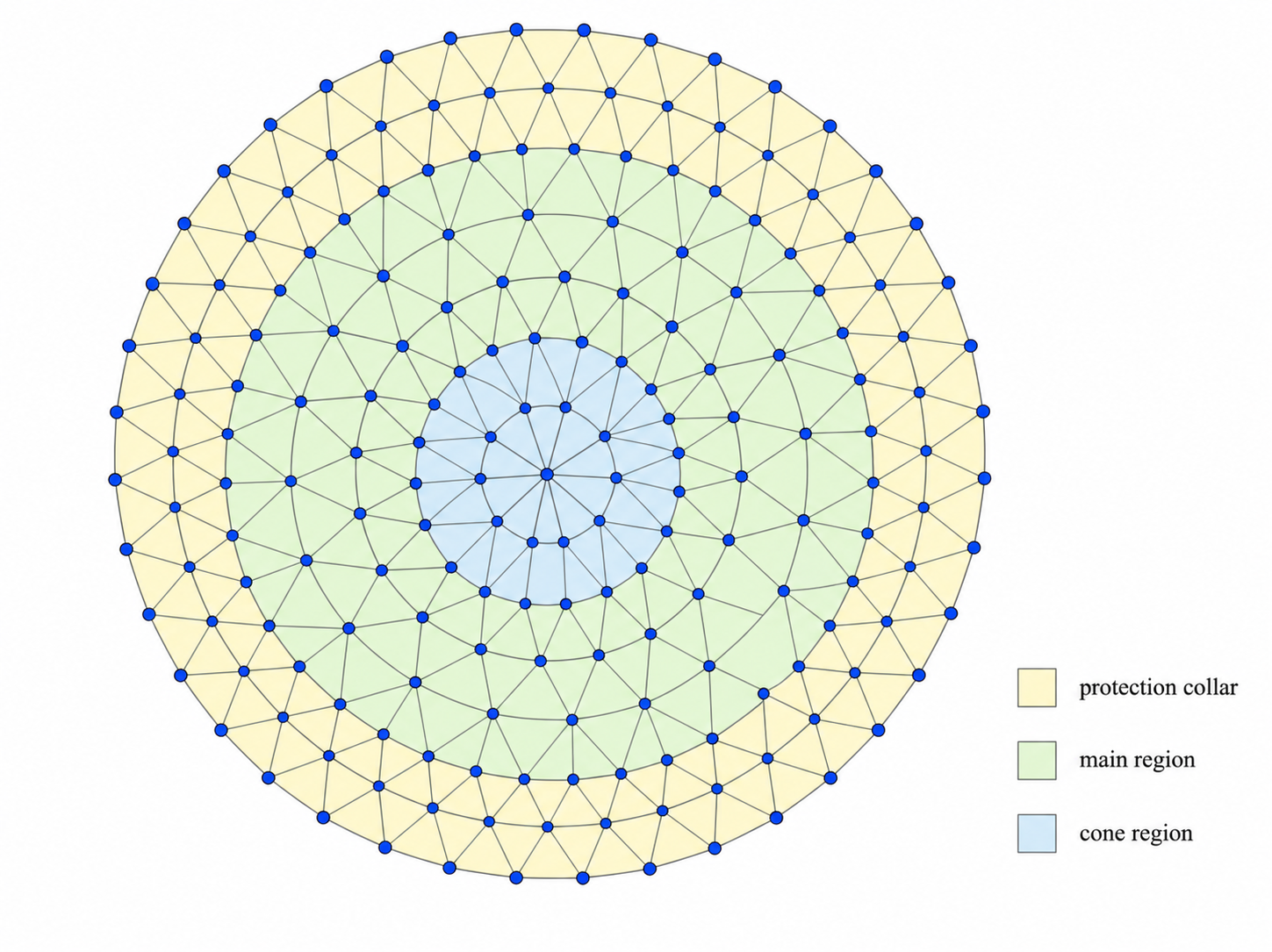}
\caption{Construction of the Discrete Isometric Filling}
\end{figure}

\section{Uniform estimates for the stepwise profile}

The main region is a stepwise approximation of the square-root profile $q(t)$.  We need estimates that hold uniformly for every possible deepest layer of a path.

For a layer $C_h$ lying in the main region, define $g_h$ to be the number of equal-length annuli in the main region passed before reaching $C_h$.  Put
\[
\tau_h=\frac{g_h}{n}.
\]
Thus $\tau_h$ records the depth inside the main region in units of $n$; it ignores the $B=\lceil \sqrt n\rceil=o(n)$ transition annuli, whose total effect is lower order.

\begin{lemma}[Uniform step-profile estimates]\label{lem:uniform}
Fix $\rho,\eta>0$.  There exists a sequence $\eps_n\to0$ such that for every path entering the main region but not the central cone, if $C_h$ is its deepest layer, then
\[
\left|2h-(2\rho n+2\tau_h n)\right|\le \eps_n n,
\]
\[
\left|\frac{m_h}{n}-q(\tau_h)\right|\le \eps_n,
\]
and
\[
\left|\mathcal A_h-\bigl(\rho n+nI(\tau_h)\bigr)\right|\le \eps_n n,
\qquad
I(t)=\int_0^t\frac{du}{q(u)}.
\]
The same $\eps_n$ works for all possible deepest layers $h$.
\end{lemma}

\begin{proof}
The protective collar has
\[
w=\rho n+O(1)
\]
layers.  Before reaching $C_h$, the number of transition annuli in the main region is at most $B=O(\sqrt n)=o(n)$.  Therefore
\[
h=w+g_h+O(B)+O(1),
\]
uniformly in $h$, which gives
\[
2h=2\rho n+2\tau_h n+o(n).
\]

Next estimate $m_h$.  If $h$ lies in the $b$-th equal-length block, then $m_h=M_b$.  If $h$ lies at a transition endpoint, then $m_h$ is one of $M_b$ and $M_{b+1}$.  Hence
\[
\frac{m_h}{n}=q(t_b)+O\left(\frac1n\right)
\quad\text{or}\quad
\frac{m_h}{n}=q(t_{b+1})+O\left(\frac1n\right).
\]
The block width is $\Delta_n=t_\eta/B$, and the accumulated discrepancy coming from floors and from transition annuli is $O(B/n)$.  Thus $\tau_h$ differs from the relevant endpoint $t_b$ or $t_{b+1}$ by at most
\[
\Delta_n+O\left(\frac{B}{n}\right)=o(1).
\]
Since $q(t)\ge\eta$ on $[0,t_\eta]$, the function $q$ is Lipschitz on this interval.  Therefore
\[
\frac{m_h}{n}=q(\tau_h)+o(1)
\]
uniformly in $h$.

It remains to estimate the accumulated drift $\mathcal A_h$.  The protective collar contributes exactly
\[
w=\rho n+O(1),
\]
because each collar annulus has length $n$ and contributes $2\cdot n/(2n)=1$ to $\mathcal A_h$.

Now consider the equal-length annuli in the main region.  Let $\ell_b(h)$ be the number of equal-length annuli in the $b$-th block passed before reaching $C_h$.  Each such annulus has length $M_b$, so each contributes
\[
2\cdot \frac{n}{2M_b}=\frac{n}{M_b}
\]
to $\mathcal A_h$.  Therefore the equal-length main-region contribution is
\[
\sum_b \ell_b(h)\frac{n}{M_b}.
\]
Since
\[
M_b=nq(t_b)+O(1)
\]
and $q(t_b)\ge\eta$, we have
\[
\frac{n}{M_b}=\frac1{q(t_b)}+O\left(\frac1n\right)
\]
uniformly in $b$.  Because $\sum_b\ell_b(h)=O(n)$, it follows that
\[
\frac1n\sum_b \ell_b(h)\frac{n}{M_b}
=
\sum_b\frac{\ell_b(h)}{n}\frac1{q(t_b)}+o(1).
\]
The last sum is a partial left-endpoint Riemann sum for the function $1/q$ over the interval $[0,\tau_h]$.  Since $q\ge\eta$ on $[0,t_\eta]$, the function $1/q$ is uniformly continuous there.  The mesh size $\Delta_n=t_\eta/B$ tends to $0$, and hence these partial Riemann sums converge uniformly in $h$ to
\[
I(\tau_h)=\int_0^{\tau_h}\frac{du}{q(u)}.
\]
Thus the equal-length main-region contribution is
\[
nI(\tau_h)+o(n).
\]

Finally, there are at most $B$ transition annuli before any layer $C_h$.  Every main-region cycle has length at least $\eta n$.  Hence each transition annulus contributes at most
\[
2\cdot \frac{n}{\eta n}=\frac{2}{\eta}
\]
to $\mathcal A_h$.  The total transition contribution is therefore
\[
O\left(\frac{B}{\eta}\right)=o(n),
\]
with $\eta$ fixed.  Combining the collar, equal-block, and transition contributions gives
\[
\mathcal A_h=\rho n+nI(\tau_h)+o(n)
\]
uniformly in $h$.
\end{proof}

\section{The core inequality in edge units}

The next lemma is the analytic heart of the construction.  It is stated directly in edge-count units.

\begin{lemma}[Core inequality]\label{lem:core}
Let
\[
q(t)=\sqrt{1-4t},
\qquad
I(t)=\int_0^t\frac{du}{q(u)}.
\]
Then, for every $0\le t\le t_\eta$ and every $0\le S\le n/2$,
\[
2tn+q(t)\bigl(S-nI(t)\bigr)_+\ge S.
\]
\end{lemma}

\begin{proof}
Write $q=q(t)=\sqrt{1-4t}$.  Then
\[
2t=\frac{1-q^2}{2},
\qquad
I(t)=\int_0^t\frac{du}{\sqrt{1-4u}}=\frac{1-q}{2}.
\]
If $S\le nI(t)$, then it is enough to show $2tn\ge S$.  This follows from
\[
2t-I(t)=\frac{1-q^2}{2}-\frac{1-q}{2}=\frac{q(1-q)}{2}\ge0.
\]

Now suppose $S\ge nI(t)$.  We need to prove
\[
2tn+q(S-nI(t))\ge S.
\]
The left-hand side minus $S$ is
\[
2tn-qnI(t)-(1-q)S.
\]
As a function of $S$, this is decreasing.  Hence the worst case is $S=n/2$.  At $S=n/2$ we get
\[
2tn+q\left(\frac n2-nI(t)\right)-\frac n2
=n\left(\frac{1-q^2}{2}+q\cdot\frac q2-\frac12\right)=0.
\]
Thus the inequality holds for all $0\le S\le n/2$.
\end{proof}

\section{Distance estimate}

Let $x,y\in C_0$ and write
\[
L=d_{C_0}(x,y),
\qquad
0\le L\le \frac n2.
\]
We prove that every edge path $P:x\to y$ in $K_n$ has length at least $L$.

\subsection{Paths not entering the main region}

Suppose $P$ reaches deepest layer $C_h$ with $h\le w$.  All annuli involved are equal-length annuli of length $n$.  Therefore $\mathcal A_h=h$, and Lemma \ref{lem:drift} gives
\[
|P|\ge 2h+(L-h)_+.
\]
If $L\ge h$, then the right-hand side is $L+h\ge L$.  If $L<h$, then it is $2h>L$.  Thus $|P|\ge L$.

\subsection{Paths entering the main region but not the cone}

Assume $P$ enters the main region but does not enter the central cone.  By Lemmas \ref{lem:drift} and \ref{lem:uniform}, for the deepest layer $C_h$ there is $\tau=\tau_h\in[0,t_\eta]$ such that
\[
|P|
\ge
2\rho n+2\tau n
+q(\tau)\bigl(L-\rho n-nI(\tau)\bigr)_+
-\eps_n n,
\]
where $\eps_n\to0$ for fixed $\rho,\eta$.

If $L\le \rho n$, then entering the main region already forces a round trip across the protective collar, so
\[
|P|\ge 2\rho n\ge L.
\]
Now assume $L>\rho n$ and set
\[
S=L-\rho n.
\]
Then $0\le S\le n/2$.  By Lemma \ref{lem:core},
\[
2\tau n+q(\tau)(S-nI(\tau))_+\ge S.
\]
Therefore
\[
|P|
\ge
2\rho n+S-\eps_n n
=L+\rho n-\eps_n n.
\]
Since $\rho>0$ is fixed, for all sufficiently large $n$ we have $\eps_n\le \rho/2$, and hence
\[
|P|\ge L+\frac{\rho n}{2}\ge L.
\]

\subsection{Paths entering the central cone}

If $P$ enters the central cone, then it must cross the protective collar and the entire main region in and out.  Ignoring transition annuli only weakens the lower bound, so
\[
|P|
\ge
2w+2\sum_{b=0}^{B-1}L_b.
\]
Since
\[
w=\rho n+O(1),
\qquad
\sum_{b=0}^{B-1}L_b=nt_\eta+O(B),
\]
we obtain
\[
\frac{|P|}{n}
\ge
2\rho+2t_\eta-o(1)
=2\rho+\frac{1-\eta^2}{2}-o(1).
\]
Because $\eta^2<\rho$, the last expression is
\[
\frac12+2\rho-\frac{\eta^2}{2}-o(1)>\frac12+\frac{3\rho}{2}-o(1).
\]
For all sufficiently large $n$, this is greater than $1/2$.  Since $L\le n/2$, such a path has length strictly larger than $L$.

We have shown that for every pair of boundary vertices $x,y$,
\[
d_{K_n}(x,y)\ge d_{C_n}(x,y).
\]
The reverse inequality follows by walking along the boundary cycle itself.  Therefore
\[
d_{K_n}(x,y)=d_{C_n}(x,y)
\]
for all boundary vertices $x,y$.  Hence $K_n$ is an isometric filling of $C_n$.

\section{Counting vertices}

We now estimate the number of vertices.  The protective collar contributes
\[
(w+1)n=\rho n^2+o(n^2)
\]
vertices, where the extra $+1$ accounts for the original boundary cycle and is only $O(n)$.

The main contribution comes from the equal-length blocks:
\[
\sum_{b=0}^{B-1}L_bM_b.
\]
The transition cycles and block endpoints contribute only $O(Bn)=o(n^2)$ vertices, since $B=o(n)$ and every cycle length is at most $n$.

We now justify the Riemann-sum limit carefully.  Recall that
\[
\Delta_n=\frac{t_\eta}{B},
\qquad
L_b=\lfloor n\Delta_n\rfloor,
\qquad
M_b=\lceil nq(t_b)\rceil.
\]
Thus, uniformly in $b$,
\[
L_b=n\Delta_n+O(1),
\qquad
M_b=nq(t_b)+O(1).
\]
Therefore
\[
\frac{L_bM_b}{n^2}
=
\Delta_n q(t_b)
+O\left(\frac{\Delta_n}{n}\right)
+O\left(\frac1n\right)
+O\left(\frac1{n^2}\right).
\]
Summing over $b=0,\dots,B-1$ gives
\[
\frac1{n^2}\sum_{b=0}^{B-1}L_bM_b
=
\sum_{b=0}^{B-1}\Delta_n q(t_b)+o(1),
\]
because
\[
\sum_{b=0}^{B-1}\Delta_n=t_\eta,
\qquad
\frac{B}{n}\to0.
\]
Since $B\to\infty$, the mesh size $\Delta_n=t_\eta/B$ tends to $0$.  The function $q$ is continuous on $[0,t_\eta]$, so the left endpoint Riemann sums converge:
\[
\sum_{b=0}^{B-1}\Delta_n q(t_b)
\longrightarrow
\int_0^{t_\eta}q(t)\,dt.
\]
Hence
\[
\frac1{n^2}\sum_{b=0}^{B-1}L_bM_b
\longrightarrow
\int_0^{t_\eta}q(t)\,dt.
\]
For $q(t)=\sqrt{1-4t}$ and $t_\eta=(1-\eta^2)/4$, this integral is
\[
\int_0^{t_\eta}\sqrt{1-4t}\,dt
=\frac14\int_{\eta^2}^{1}u^{1/2}\,du
=\frac16(1-\eta^3).
\]
The cone cap adds only one new vertex.  Therefore, for fixed $\rho$ and $\eta$,
\[
\limsup_{n\to\infty}\frac{|V(K_n)|}{n^2}
\le
\rho+\frac16(1-\eta^3)
\le
\rho+\frac16.
\]

For every $\rho>0$, choose $0<\eta<\sqrt\rho$.  The construction above gives isometric fillings satisfying
\[
\limsup_{n\to\infty}\frac{|V(K_n)|}{n^2}
\le \frac16+\rho.
\]
Since $\rho>0$ is arbitrary,
\[
D^*\le\frac16.
\]
This proves Theorem \ref{thm:main}.

\end{document}